\documentclass{article}

\usepackage{amssymb}
\usepackage{amsmath}
\usepackage{amsthm}

\newtheorem{theorem}{Theorem}
\newtheorem{definition}[theorem]{Definition}
\newtheorem{proposition}[theorem]{Proposition}
\newtheorem{remark}[theorem]{Remark}

\begin{document}

\title{Necessary conditions for linear noncooperative
$N$-player delta differential games on time
scales\thanks{Partially presented at the \emph{Fifth Symposium on
Nonlinear Analysis} (SNA 2007), Toru\'{n}, Poland, September
10-14, 2007.}}

\author{Nat\'{a}lia Martins \and Delfim F. M. Torres}

\date{Department of Mathematics\\
University of Aveiro\\
3810-193 Aveiro, Portugal\\[0.3cm]
\texttt{\{natalia,delfim\}@ua.pt}}

\maketitle


\begin{abstract}
We present necessary conditions for linear noncooperative
$N$-player delta dynamic games on a generic time scale.
Necessary conditions for an open-loop Nash-equilibrium and for a
memoryless perfect state Nash-equilibrium are proved.
\end{abstract}


\smallskip

\noindent \textbf{Mathematics Subject Classification 2000:} 49N70,
91A10, 39A10.

\smallskip


\smallskip

\noindent \textbf{Keywords:} time scales, delta differential
games, calculus of variations on time scales.


\section{Introduction}

In 1988 Stephan Hilger developed in his PhD thesis \cite{Hilger2}
the theory of \emph{time scales}. The set $\mathbb{R}$ of real
numbers and the set $\mathbb{Z}$ of integers     are (trivial)
examples of time scales. When a result is proved in a general time
scale $\mathbb{T}$, one unifies both continuous and discrete
analysis. Moreover, since there are infinitely many other time
scales, a much more general result is proved. For this reason, one
can say that the two main features of the theory of time scales
are \emph{unification} and \emph{extension}.

\emph{Differential game theory} is a relatively new area of
Mathematics, initiated in the fifties of the XX century with the
works of Rufus Isaacs, that founds many applications in different
fields such as economics, politics, and artificial intelligence.
The theory has been studied in the context of classical
Mathematics, in discrete or continuous time \cite{livroJank,Leitmann}.
We trust that it is also possible (with some advantages) to present
(delta) differential games in the generic context of time scales.
To the best of the authors knowledge, this paper represents the
first attempt to provide a delta differential game theory on a
generic time scale.

The paper is organized as follows. In Section~\ref{sec:2} we
review some basic definitions and results from the calculus
on time scales. In Section~\ref{sec:3},
necessary conditions for a weak local
minimizer of a Lagrange problem on time scales
(Theorem~\ref{PMP-time-scales}) are presented, while
Section~\ref{sec:4} recalls a result that guarantees the
uniqueness of the forward solution for a special initial valued
problem on time scales (Theorem~\ref{Zbigniew+Ewa}). In
Section~\ref{sec:5} we introduce the definition of a $N$-player
delta differential game, the notion of Nash-equilibrium, and two
types of information structure in a game: open-loop and memoryless
perfect state. The main results of the paper appear then in
Section~\ref{sec:6}, where necessary conditions for a linear
open-loop Nash-equilibrium and for a linear memoryless perfect
state Nash-equilibrium are proved.


\section{Calculus on time scales}
\label{sec:2}

We briefly present here the necessary concepts and results from
the theory of time scales (\textrm{cf.}
\cite{Bohner-Peterson1,Bohner-Peterson2,Hilger} and references
therein). As usual, $\mathbb{R}$, $\mathbb{Z}$  and $\mathbb{N}$
denote, respectively, the set of real, integer, and natural
numbers.

A \emph{time scale} $\mathbb{T}$ is an arbitrary nonempty closed
subset of $\mathbb{R}$. Thus, $\mathbb{R}$, $\mathbb{Z}$ and
$\mathbb{N}$ are examples of times scales. Other examples of times
scales can be $h\mathbb{Z}$, for some $h>0$, $[1,4] \bigcup
\mathbb{N}$, or the Cantor set. We assume that a time scale
$\mathbb{T}$ has the topology that it inherits from the real
numbers with the standard topology.

The \emph{forward jump operator}
$\sigma:\mathbb{T}\rightarrow\mathbb{T}$ is defined by
$$\sigma(t)=\inf{\{s\in\mathbb{T}:s>t}\}$$
if $t\neq \sup \mathbb{T}$ and $\sigma(\sup \mathbb{T})=\sup
\mathbb{T}$. The \emph{backward jump operator}
$\rho:\mathbb{T}\rightarrow\mathbb{T}$ is defined by
$$\rho(t)=\sup{\{s\in\mathbb{T}:s<t}\}$$
if $t\neq \inf \mathbb{T}$ and $\rho(\inf \mathbb{T})=\inf
\mathbb{T}$.

A point $t\in\mathbb{T}$ is called \emph{right-dense},
\emph{right-scattered}, \emph{left-dense} or \emph{left-scattered}
if $\sigma(t)=t$, $\sigma(t)>t$, $\rho(t)=t$ or $\rho(t)<t$,
respectively.

The \emph{graininess function}
$\mu:\mathbb{T}\rightarrow[0,\infty)$ is defined by
$$\mu(t)=\sigma(t)-t,\mbox{ for all $t\in\mathbb{T}$.}$$
For a given instant $t$, $\mu(t)$ measures the distance of $t$ to
its right neighbor.

It is clear that when $\mathbb{T}=\mathbb{R}$, then for any $t \in
\mathbb{R}$, $\sigma(t)=t=\rho(t)$ and $\mu=0$. When
$\mathbb{T}=\mathbb{Z}$, for any $t \in \mathbb{Z}$,
$\sigma(t)=t+1$, $\rho(t)=t-1$, and $\mu=1$.

In order to introduce the definition of \emph{delta derivative},
we define the set
$$\mathbb{T}^k:=\{t \in \mathbb{T}: t  \ \ \mbox{is nonmaximal or left-dense} \} \, .$$
Thus, $\mathbb{T}^k$ is obtained from $\mathbb{T}$ by removing its
maximal point if this point exists and is left-scattered.

We say that a function $f:\mathbb{T}\rightarrow\mathbb{R}$ is
\emph{delta differentiable} at $t\in\mathbb{T}^k$ if there is a
number $f^{\Delta}(t)$ such that for all $\varepsilon>0$ there
exists a neighborhood $U$ of $t$ (\textrm{i.e.},
$U=]t-\delta,t+\delta[\cap\mathbb{T}$ for some $\delta>0$) such
that
$$
|f(\sigma(t))-f(s)-f^{\Delta}(t)(\sigma(t)-s)|
\leq\varepsilon|\sigma(t)-s|,\mbox{ for all $s\in U$}.
$$
We call $f^{\Delta}(t)$ the \emph{delta derivative} of $f$ at $t$.
Moreover, we say that $f$ is \emph{delta differentiable} on
$\mathbb{T}^k$ provided $f^{\Delta}(t)$ exists for all $t \in
\mathbb{T}^k$.

We note that when $\mathbb{T}=\mathbb{R}$, then $f:\mathbb{R}
\rightarrow \mathbb{R}$ is delta differentiable at $t \in
\mathbb{R}$ if and only if $$\displaystyle f^{\Delta}(t)=
\lim_{s\rightarrow t}\frac{f(t)-f(s)}{t-s}$$ exists, \textrm{i.e.}
if and only if $f$ is differentiable in the ordinary sense at $t$.
When $\mathbb{T}=\mathbb{Z}$, then $f:\mathbb{Z} \rightarrow
\mathbb{R}$ is always delta differentiable at $t \in \mathbb{Z}$
and
$$
f^{\Delta}(t)=\frac{f(\sigma(t))-f(t)}{\sigma(t)-t}= f(t+1)-f(t) \, .
$$
Hence, for $\mathbb{T}=\mathbb{Z}$ the delta derivative of $f$,
$f^{\Delta}$, coincides with the usual forward difference $\Delta f$.

It is clear that if $f$ is constant, then $f^{\Delta}=0$; if
$f(t)=k t$ for some constant $k$, then $f^{\Delta}=k$.

For delta differentiable functions $f$ and $g$, the next formulas
hold:
\begin{align*}
f^\sigma(t)&=f(t)+\mu(t)f^\Delta(t) \, ,\\
(fg)^\Delta(t)&=f^\Delta(t)g^\sigma(t)+f(t)g^\Delta(t)\\
&=f^\Delta(t)g(t)+f^\sigma(t)g^\Delta(t),
\end{align*}
where we abbreviate $f\circ\sigma$ by $f^\sigma$.

Delta derivatives of higher order are defined in the standard way:
we define the $r^{th}$-\emph{delta derivative} ($r\in\mathbb{N}$)
of $f$ to be the function
$f^{\Delta^r}:\mathbb{T}^{k^r}\rightarrow\mathbb{R}$, provided
$f^{\Delta^{r-1}}$ is delta differentiable on
$\mathbb{T}^{k^r}:=\left(\mathbb{T}^{k^{r-1}}\right)^k$.

The class of continuous functions on $\mathbb{T}$ is too small for
a convenient theory of integration. For our purposes, it is enough
to define the notion of integral in the class of rd-continuous
functions. For a more general theory of integration on time
scales, we refer the reader to \cite{Bohner-Peterson2}. A function
$f:\mathbb{T}\rightarrow\mathbb{R}$ is called \emph{rd-continuous}
if it is continuous at right-dense points and if its left-sided
limit exists (finite) at left-dense points. For
$\mathbb{T}=\mathbb{R}$ rd-continuity coincides with continuity.

We denote the set of all rd-continuous functions
$f:\mathbb{T}\rightarrow\mathbb{R}$ by
$C_{\textrm{rd}}(\mathbb{T}, \mathbb{R})$, and the set of all
delta differentiable functions with rd-continuous derivative by
$C_{\textrm{rd}}^1(\mathbb{T}, \mathbb{R})$.

It can be shown that every rd-continuous function $f$ possess an
\emph{antiderivative}, \textrm{i.e.} there exists a function $F$
with $F^\Delta=f$, and in this case the \emph{delta integral} is
defined by
$$\int_{a}^{b}f(t)\Delta t:=F(b)-F(a) \ \ \ \mbox{for all } a,b
\in \mathbb{T}.$$ This integral has the following property:
$$
\int_t^{\sigma(t)}f(\tau)\Delta\tau=\mu(t)f(t) \, .
$$
Let $a, b \in \mathbb{T}$ and $f\in C_{\textrm{rd}}(\mathbb{T},
\mathbb{R})$. It is easy to prove that
\begin{enumerate}
\item for $\mathbb{T}=\mathbb{R}$, $\int_{a}^{b}f(t)\Delta t =
\int_{a}^{b}f(t) dt$, where the integral on the right hand side is
the usual Riemann integral;

\item for $\mathbb{T}=\mathbb{Z}$, $\int_{a}^{b}f(t)\Delta t =
\sum_{t=a}^{b-1}f(t)$ if $a<b$, $\int_{a}^{b}f(t)\Delta t=0$ if
$a=b$, and $\int_{a}^{b}f(t)\Delta t = - \sum_{t=b}^{a-1}f(t)$ if
$a>b$.
\end{enumerate}

Next, we present the integration by parts formulas: if
$a,b\in\mathbb{T}$ and $f,g\in C_{\textrm{rd}}(\mathbb{T},
\mathbb{R})$, then
\begin{enumerate}
\item $\int_{a}^{b}f(\sigma(t))g^{\Delta}(t)\Delta t
=\left[(fg)(t)\right]_{t=a}^{t=b}-\int_{a}^{b}f^{\Delta}(t)g(t)\Delta
t$;

\item $\int_{a}^{b}f(t)g^{\Delta}(t)\Delta t
=\left[(fg)(t)\right]_{t=a}^{t=b}-\int_{a}^{b}f^{\Delta}(t)g(\sigma(t))\Delta
t$.
\end{enumerate}

Similarly to classical calculus, we say that
$f:\mathbb{T}\rightarrow\mathbb{R}^n$ is a \emph{rd-continuous}
function if each component of $f$,
$f_i:\mathbb{T}\rightarrow\mathbb{R}$, is a rd-continuous
function. The set of all such functions is denoted by
$C_{rd}(\mathbb{T}, \mathbb{R}^{n})$. The set
$C_{rd}^1(\mathbb{T}, \mathbb{R}^{n})$ is defined in the usual
way.


\section{Lagrange problem on time scales}
\label{sec:3}

Let $a, b \in \mathbb{T}$ such that $a<b$. In what follows we
denote by $[a,b]$ the set $\{t \in \mathbb{T}: a \leq t \leq b\}$.
Consider the following Lagrange  problem with delta differential
side condition:
\begin{equation}
\label{Lagrange Problem}
\begin{gathered}
J[x(\cdot),u(\cdot)] = \int_a^b L(t, x(t), u(t))\Delta
t\longrightarrow \min \, ,\\
x^{\Delta}(t) =  \varphi(t, x(t), u(t)) \, , \quad  t\in
\mathbb{T}^k \, , \\
x(a) = x_a \, ,
\end{gathered}
\end{equation}
where we assume that
\begin{itemize}
\item $L:\mathbb{T} \times \mathbb{R}^n \times
\mathbb{R}^m\rightarrow \mathbb{R}$, $\varphi:\mathbb{T} \times
\mathbb{R}^n \times \mathbb{R}^m\rightarrow \mathbb{R}^n$, $(x,u)
\rightarrow L(t,x,u)$ and $(x,u) \rightarrow \varphi(t,x,u)$ are
$C^1$-functions of $x$ and $u$ for each $t$;

\item $x(\cdot) \in C_{rd}^{1}(\mathbb{T}, \mathbb{R}^n)$ and
$u(\cdot) \in C_{rd}(\mathbb{T}, \mathbb{R}^m)$, $m\leq n$;

\item for each control function $u(\cdot)$ there exists a unique
forward solution $x(\cdot)$ of the initial value problem
$x^{\Delta}(t) = \varphi(t, x(t), u(t))$,
$x(a)=x_a$.\footnote{In the linear case,
conditions guaranteing the existence and uniqueness of forward
solutions are easy to obtain \cite{Zbigniew+Ewa}
-- \textrm{cf.} Section~\ref{sec:4}.}
\end{itemize}

We borrow from \cite{Delfim-Rui} the definition of admissible pair
and the definition of weak local minimizer for problem
(\ref{Lagrange Problem}). The reader interested in the calculus of
variations on time scales is referred to
\cite{Tangier-with-Rui,Delfim-Rui2} and references therein.

\begin{definition} The pair $(x_{*}(\cdot),u_{*}(\cdot))$ is said
to be \emph{admissible} for problem (\ref{Lagrange Problem}) if
\begin{enumerate}
\item  $x_{*}(\cdot) \in C_{rd}^{1}(\mathbb{T}, \mathbb{R}^n)$ and
$u_{*}(\cdot) \in C_{rd}(\mathbb{T}, \mathbb{R}^m)$;

\item  $x_*^{\Delta}(t) = \varphi(t, x_*(t), u_*(t))$ and
$x_{*}(a)=x_a$.
\end{enumerate}
\end{definition}

In order to introduce the notion of weak local minimizer for
problem (\ref{Lagrange Problem}), we define the following norms in
$C_{rd}(\mathbb{T}, \mathbb{R}^n)$ and $C_{rd}^1(\mathbb{T},
\mathbb{R}^n)$:
$$
||y||_{\infty}:= \sup_{t \in\mathbb{T}}
\parallel y(t)\parallel \, , \quad y \in C_{rd}(\mathbb{T},
\mathbb{R}^n) \, ,
$$
and
$$
||z||_{1,\infty} := \sup_{t \in\mathbb{T}^k}
\parallel z(t)\parallel
+ \sup_{t \in\mathbb{T}^k} \parallel z^{\Delta}(t)\parallel \, ,
\quad z \in C_{rd}^1(\mathbb{T}, \mathbb{R}^n) \, ,
$$
where $\parallel \cdot \parallel$ denotes the Euclidean norm in
$\mathbb{R}^n$.

\begin{definition}
An admissible pair $(x_{*}(\cdot),u_{*}(\cdot))$ is said to be a
\emph{weak local minimizer} for problem (\ref{Lagrange Problem})
if there exists $r \in \mathbb{R}^+$ such that
$$
J[x_{*}(\cdot),u_{*}(\cdot)] \leq J[x(\cdot),u(\cdot)]
$$
for all admissible pairs $(x(\cdot),u(\cdot))$ satisfying
$$
\parallel x-x_{*}\parallel_{1,\infty} + \parallel
u-u_{*}\parallel_{\infty}< r.
$$
\end{definition}

Theorem~\ref{PMP-time-scales} gives necessary conditions for a
pair $(x_{*}(\cdot),u_{*}(\cdot))$ to be a weak local minimizer of
the Lagrange problem (\ref{Lagrange Problem}).

\begin{theorem}[\cite{Delfim-Rui}]
\label{PMP-time-scales} If $(x_{*}(\cdot),u_{*}(\cdot))$ is a weak
local minimizer of problem (\ref{Lagrange Problem}), then there
exists a multiplier $\psi_{*}(\cdot):
\mathbb{T}\rightarrow\mathbb{R}^n$ that is  delta differentiable
on $\mathbb{T}^k$, such that
\begin{equation}
\label{x delta}
\begin{gathered}
x_{*}^\Delta(t)=
\mathcal{H}_{\psi^{\sigma}}(t,x_{*}(t),u_{*}(t),
\psi_{*}^{\sigma}(t)) \, ,\\
\psi_{*}^\Delta(t)= -
\mathcal{H}_{x}(t,x_{*}(t),u_{*}(t), \psi_{*}^{\sigma}(t)) \, ,\\
 \mathcal{H}_{u}(t,x_{*}(t),u_{*}(t),
\psi_{*}^{\sigma}(t))=0 \, ,\\
\psi_{*}(b)=0 \, ,
\end{gathered}
\end{equation}
for all $t \in \mathbb{T}^k$, where the Hamiltonian function
$\mathcal{H}$ is defined by
\begin{equation}
\label{eq:Ham}
\mathcal{H}(t, x, u,  \psi^\sigma):=  L(t,x,u) + \psi^\sigma\cdot
\varphi(t,x,u) \, .
\end{equation}
\end{theorem}

\begin{remark}
From definition (\ref{eq:Ham}) of the Hamiltonian function $\mathcal{H}$,
it follows that the first condition in (\ref{x delta}) holds
for any admissible pair $(x_{*}(\cdot),u_{*}(\cdot))$ of problem
(\ref{Lagrange Problem}).
\end{remark}

\begin{remark}
For the time scale $\mathbb{T}= \mathbb{R}$,
Theorem~\ref{PMP-time-scales} is a particular case of Pontryagin's
Maximum Principle \cite{Pontryagin}.
\end{remark}


\section{Linear systems on times scales}
\label{sec:4}

Let us consider the following initial value problem on a time
scale $\mathbb{T}$:

\begin{equation}
\label{IVP1}
\begin{cases}
x^{\Delta}(t) =  A x(t) + f(t) \, , \\
x(a) = x_a \, ,
\end{cases}
\end{equation}
where $A$ is a constant $n\times n$ matrix, $f: \mathbb{T}
\rightarrow\mathbb{R}^n$ is a rd-continuous function, $a \in
\mathbb{T}$ and $x_a \in \mathbb{R}^n$.

Similar to control theory \cite{Zbigniew+Ewa},
in dynamic game theory we are interested in forward solutions.
The purpose of this section is to present conditions assuring
problem (\ref{IVP1}) to have a unique forward solution.

\begin{proposition}[\cite{Bohner-Peterson1}]
The initial value problem
\begin{equation*}
\label{IVP3}
\begin{cases}
x^{\Delta}(t) =  A x(t) \, ,\\
x(a) = x_a \, ,
\end{cases}
\end{equation*}
where $A$ is a constant $n\times n$ matrix, $a \in \mathbb{T}$ and
$x_a \in \mathbb{R}^n$, has a unique forward solution.
\end{proposition}

The \emph{matrix exponential function} (also known as the
\emph{fundamental matrix solution}) at $a$ for the matrix $A$, is
defined as the unique forward solution of the matrix differential
equation
$$X^{\Delta}(t)= AX(t),$$ with the initial condition $X(a)=I$,
where $I$ denotes the $n\times n$ identity matrix. Its value at
$t$ is denoted by $e_A(t,a)$.

\begin{theorem}[\textrm{cf.} \cite{Zbigniew+Ewa}]
\label{Zbigniew+Ewa} The initial value problem (\ref{IVP1}) has a
unique forward solution of the form
$$
x(t)= e_A(t,a)x_a + \int_a^t e_A(t,\sigma(s))f(s) \Delta s \, .
$$
\end{theorem}


\section{$N$-player delta differential games}
\label{sec:5}

The (classical) term ``$N$-player dynamic game''
is applied to a group of problems in applied mathematics that
possess certain characteristics related with conflict problems.
The main ingredients in a $N$-player dynamic game
are the players, the control variables, the state variables, and
the cost functionals/functions. The relation between state and
control variables is given by a differential/difference equation.
Two types of games can be considered: cooperative or
noncooperative games. In this paper we shall restrict ourselves to
noncooperative games. In a noncooperative game the players act
independently in the pursuit of their own best interest, each
player desiring to attain the smallest possible cost.

Following Jank \cite{Jank} (see also \cite{livroJank}),
we introduce the notion of a $N$-player delta
differential game (noncooperative dynamic game
in the context of time scales).

\begin{definition}
\label{Delta-differential games}
Let $N \in \mathbb{N}$. We say that
$$\Gamma_N =(\mathcal{T}, X, U_i, \mathcal{U}_i,
\sigma_i, f, x_a, \eta^i, J^i)_{i=1,2,\ldots, N}$$ is a
$N$-\emph{player delta differential game} if:
\begin{enumerate}
\item $\mathcal{T}$ is a closed nonempty  interval of $\mathbb{T}$
($\mathcal{T}$ is called the \emph{time horizon});

\item $X$ is a finite dimensional Euclidean space ($X$ is called
the \emph{state} or \emph{phase} space);

\item $U_i$ is a finite dimensional Euclidean space ($U_i$ is
called the \emph{control value space} of the $i$-th player);

\item $\sigma_i$ is a subset of a set of mappings
$$
\{\gamma^i | \gamma^i: \mathcal{T}\times \mathcal{P}(X)
\rightarrow U_i\}
$$
($\gamma^i$ is called a \emph{strategy} of the i-th player while
$\sigma_i$ is called the set of \emph{possible strategies} of the
i-th player);

\item $x$ is a mapping from $\mathcal{T}$ to $X$ ($x$ is called
the \emph{state variable});

\item $\eta^i :\mathcal{T} \rightarrow  \mathcal{P}(X)$ is a
mapping with the property
$$
\eta^i(t)  \subseteq \{ x(s) \  | \ a \leq s \leq t\}
$$
($\eta^i$ is called the \emph{information structure} of the i-th
player);

\item $\mathcal{U}_i= \{\gamma^i(\cdot, \eta^i(\cdot)) \ | \
\gamma^i \in  \sigma_i\}$ ($\mathcal{U}_i$ is called the
\emph{control space} or \emph{decision set} of the i-th player);

\item $f: \mathcal{T}\times X \times \mathcal{U}_1 \times \cdots
\times \mathcal{U}_N \rightarrow X$ is a mapping that describes an
outer force acting on the system by the delta differential
equation
$$
x^{\Delta}(t)= f(t,x(t),u^1(t), \ldots, u^N(t))
$$
with the initial condition $x(a)=x_a \in X$ ($u^i \in
\mathcal{U}_i$ is called the \emph{control function} or
\emph{control variable} of the $i$-th player);

\item $J^i$ is a mapping from $\mathcal{U}_1 \times \cdots \times
\mathcal{U}_N$ to $\mathbb{R}$  ($J^i$ is called the \emph{cost
functional} of the i-th player).
\end{enumerate}
\end{definition}

\begin{remark}
For each $i=1,2,\ldots,N$, the controls $u^i(\cdot) =
\gamma^i(\cdot, \eta^i(\cdot)) \in \mathcal{U}_i$ are built up
from chosen strategies $\gamma^i$ with a specific information
structure $\eta^i$.
\end{remark}

\begin{remark}
The cost functional $J^i$ of the $i$-th
player, $i \in \{1,2,\ldots,N\}$, depends on the controls of all
the $N$ players.
\end{remark}

\begin{definition} A $N$-tuple of control functions $u(\cdot)=
(u^1(\cdot), \ldots, u^N(\cdot))$ is called an \emph{admissible
control} for the system $\Gamma_N$ if $u$ is a rd-continuous
function  and there exists a unique forward solution $x \in
C^1_{rd}(\mathcal{T}, X)$ of the initial value problem
$x^{\Delta}(t)= f(t,x(t),u^1(t), \ldots, u^N(t))$, $x(a)=x_a$.
\end{definition}

We shall consider the following equilibrium concept in a
$N$-player delta differential game.

\begin{definition}
\label{Nash equilibrium} A decision $N$-tuple $u_{*}=(u_{*}^{1},
\ldots,u_{*}^{N}) \in \mathcal{U}_1 \times \cdots \times
\mathcal{U}_N$ is called a \emph{Nash-equilibrium} if for all
$i=1,2,\ldots, N$,
$$
J^i(u_{*}^{1}, \ldots,u_{*}^{N}) \leq J^i(u_{*}^{1},
\ldots,u_{*}^{i-1}, u^i, u_{*}^{i+1}, \ldots, u_{*}^{N})
$$
for all $u^i \in \mathcal{U}_i$.
\end{definition}

The meaning of the  Nash-equilibrium is the following: if the
$i$-th player unilaterally change the strategy $u_{*}^{i}$, then
his cost will increase.

We remark that the information available to the players is an
important aspect of the game. In this paper we consider open-loop
and memoryless perfect state information structures.

\begin{definition}\label{information structure} Let $\Gamma_N$ be a
$N$-player delta differential game. We say that $\Gamma_N$ has
\begin{enumerate}
\item \emph{Open-loop} (OL) information structure if
$\eta^i(t)=\{x_a\}, t \in \mathcal{T}$, for all $i \in
\{1,2,\ldots,N\}$;

\item \emph{Memoryless perfect state} (MPS) information structure
if $\eta^i(t)=\{x_a, x(t)\}$, $t \in \mathcal{T}$, for all $i \in
\{1,2,\ldots,N\}$.
\end{enumerate}
\end{definition}

Therefore, in the open-loop information structure each player know
only the initial position $x_a$, while in the memoryless perfect
state information structure each player know the phase state
$x(t)$ at each time instant $t$ as well as the initial position
$x_a$.

\begin{definition}
We say that a control $u^{i}(t):=\gamma^i(t, \eta^i(t))$ is: (i)
an \emph{OL control}, if the information structure $\eta^i$ is OL;
(ii) a \emph{MPS control}, if the information structure $\eta^i$
is MPS.
\end{definition}


\section{Main results}
\label{sec:6}

Using Theorem~\ref{PMP-time-scales}, we deduce necessary
conditions for OL and MPS Nash-equilibrium of $N$-player games
with linear delta differential equations. We deal with cost
delta-integral functionals of the following type:
\begin{equation}
\label{cost} J^i(u^1, \ldots, u^N)= \int_a^b L^i(t, x(t), u^1(t),
\ldots, u^N(t)) \Delta t \, ,
\end{equation}
where we suppose the following (from now on, $\mathbb{T}$ denotes
$\mathbb{T} \cap [a,b]$):
\begin{itemize}
\item for each $i=1,2,\ldots,N$, $L^i:\mathbb{T} \times
\mathbb{R}^n \times \mathbb{R}^{m_1} \times \ldots \times
\mathbb{R}^{m_N} \rightarrow \mathbb{R}$ is a $C^1$-function of
$x$ and $u=(u^1, \ldots, u^N)$ for each $t$; $\max
\{m_1,\ldots,m_N \} \leq n$;

\item for each $i=1,2,\ldots,N$, $u^i \in \mathcal{U}_i$ is
admissible;

\item $x$ is the forward solution of the delta differentiable
initial value problem
\begin{equation}
\label{IVP}
\begin{cases}
x^{\Delta}(t) =  Ax(t)+B^1 u^1(t) + \cdots + B^N u^N(t) \, ,\\
x(a) = x_a \, ,
\end{cases}
\end{equation}
where $A$ is a constant $n\times n$ matrix, and for each
$i=1,2,\ldots,N$, $B^i$ is a constant $n\times m_i$ matrix.
\end{itemize}

\begin{remark}
Since each $u^i \in C_{rd}(\mathbb{T},\mathbb{R}^{m_i})$, and
therefore $B^1 u^1 + \cdots +B^N u^N$ is a rd-continuous function,
by Theorem~\ref{Zbigniew+Ewa} the initial value problem
(\ref{IVP}) has a unique forward solution.
\end{remark}

\begin{theorem}[Necessary conditions for a linear OL Nash-equilibrium]
\label{Theo1} Let $\Gamma_N$ be a $N$-player delta differential
game, where the cost functional of the i-th player is given by
(\ref{cost}). If the decision $N$-tuple $(u^1_{*},\ldots,
u^N_{*})$ is an OL Nash-equilibrium of \ $\Gamma_N$, and if $x_{*}$
is the associated trajectory of the state, then there exist
delta differentiable functions
$\psi^i:\mathbb{T}\rightarrow \mathbb{R}^n$ on $\mathbb{T}^k$,
$i=1,2,\ldots, N$, such that for
$$
H^i(t,x,u^1, \ldots, u^N, (\psi^i)^{\sigma}) := L^i(t,x,u^1,
\ldots, u^N) + (\psi^i)^{\sigma} \cdot (Ax+B^1 u^1 + \cdots + B^N
u^N)
$$
one has:
\begin{enumerate}
\item $x_{*}^\Delta(t)=
H^i_{(\psi^i)^{\sigma}}(t,x_{*}(t),u^1_{*}(t),\ldots, u^N_{*}(t),
(\psi^i)^{\sigma}(t))$;

\item $x_{*}(a)=x_a$;

\item $({\psi^{i}})^\Delta(t)= -
H_{x}^i(t,x_{*}(t),u^1_{*}(t),\ldots, u^N_{*}(t),
(\psi^i)^{\sigma}(t))$;

\item $\psi^i(b)=0$;

\item $\displaystyle H^i_{u^i}(t,x_{*}(t),u^1_{*}(t),\ldots,
u^N_{*}(t), (\psi^i)^{\sigma}(t))= 0$;

\end{enumerate}
for all $i=1,2, \ldots, N$ and $t \in \mathbb{T}^k$.
\end{theorem}

\begin{proof}
Suppose that the decision $N$-tuple $(u_{*}^{1},
\ldots,u_{*}^{N})$ is an OL Nash-equilibrium of \ $\Gamma_N$ and
$x_{*}$ is the associated trajectory of the state. For each
$i=1,\ldots,N$ consider the functional
$$
\widetilde{J}^i: \mathcal{U}_i \rightarrow \mathbb{R}
$$
defined by
$$
\widetilde{J}^i(u^i):= J^i(u_{*}^{1}, \ldots,u_{*}^{i-1}, u^i,
u_{*}^{i+1}, \ldots, u_{*}^{N}).
$$
Define also
$$
f(x,u^1, \ldots, u^N):= Ax + B^1u^1 + \cdots + B^N u^N.
$$
Since $(u_{*}^{1}, \ldots,u_{*}^{N})$ is a Nash-equilibrium, then
for all $u^i \in \mathcal{U}_i$
$$
\widetilde{J}^i(u_ {*}^i) \leq \widetilde{J}^i(u^i).
$$
Therefore, $(x_{*}(\cdot), u_{*}^i(\cdot))$ is a weak local
minimizer of the problem
\begin{equation}
\label{problem}
\begin{gathered}
\widetilde{J}^i(u^i) \longrightarrow \min \, ,\\
x^{\Delta}(t) = \widetilde{f}(x(t),u^i(t)) \, , \\
x(a) = x_a \, ,
\end{gathered}
\end{equation}
where $\widetilde{f}(x,u^i) := f(x, u^1_{*},
\ldots, u^{i-1}_{*}, u^i,
u^{i+1}_{*}, \ldots, u^{N}_{*})$. Define
$$
\mathcal{H}^i(t,x,u^i,(\psi^i)^{\sigma}):= H^i(t,x,u^1_*, \ldots,
u_{*}^{i-1}, u^i, u_{*}^{i+1}, \ldots, u^N_*, (\psi^i)^{\sigma}) \, ,
$$
with $H^i$ the Hamiltonian defined in the statement of the
theorem. Applying Theorem~\ref{PMP-time-scales} to
$\mathcal{H}^i$ (note that $\widetilde{f}$ is a $C^1$-function of
$x$ and $u^i$ for each $t$), we conclude that there exists a
function $\psi^i : \mathbb{T}\rightarrow \mathbb{R}^n$ that is
delta differentiable on $\mathbb{T}^k$, such that $\psi^i(b)=0$,
\begin{equation*}
\begin{split}
x_{*}^\Delta(t) & =
\mathcal{H}^i_{({\psi^i})^{\sigma}}(t,x_*(t),u^i_*(t),(\psi^i)^{\sigma}(t))\\
&=  Ax_*(t) +
B^1u^1_*(t)+ \cdots + B^i u^i_*(t) + \cdots +B_Nu_*^N (t)\\
&= H^i_{({\psi^i})^{\sigma}}(t,x_{*}(t),u^1_{*}(t),\ldots,
u^N_{*}(t), (\psi^i)^{\sigma}(t)) \, , \\[0.3cm]
({\psi^{i}})^\Delta(t) &=
- \mathcal{H}^i_{x}(t,x_*(t),u^i_*(t),(\psi^i)^{\sigma}(t))\\
&=  - H_{x}^i(t,x_{*}(t),u^1_{*}(t),\ldots, u^N_{*}(t),
(\psi^i)^{\sigma}(t)) \, ,
\end{split}
\end{equation*}
and $\mathcal{H}^i_{u^i}(t,x_*(t),u^i_*(t),(\psi^i)^{\sigma}(t))
= H^i_{u^i}(t,x_{*}(t),u^1_{*}(t),\ldots,
u^N_{*}(t), (\psi^i)^{\sigma}(t))= 0$.
\end{proof}

\begin{remark}
Notice that conditions \emph{1} and \emph{2} of
Theorem~\ref{Theo1} simply assert that $x_*$ is a solution of the
initial value problem (\ref{IVP}).
\end{remark}

\begin{theorem}[Necessary conditions for a linear MPS
Nash-equilibrium] Let $\Gamma_N$ be a $N$-player delta
differential game, where the cost functional of the i-th player is
given by (\ref{cost}). If the decision $N$-tuple
$(u_{*}^{1}(\cdot), \ldots,u_{*}^{N}(\cdot))$,
given by $(\gamma^1_*(\cdot,x_a,x_*(\cdot)),
\ldots, \gamma^N_*(\cdot,x_a,x_*(\cdot)))$, is a MPS
Nash-equilibrium of \ $\Gamma_N$ and if $x_{*}(\cdot)$ is the
associated trajectory of the state, then there exist
delta differentiable functions
$\psi^i:\mathbb{T}\rightarrow \mathbb{R}^n$
on $\mathbb{T}^k$, $i=1,2,\ldots, N$, such that for
$$
H^i(t,x,u^1, \ldots, u^N,(\psi^i)^{\sigma}) := L^i(t,x,u^1,
\ldots, u^N)
+ (\psi^i)^{\sigma} \cdot (Ax+B^1 u^1 + \cdots +B^N u^N)
$$
one has:
\begin{enumerate}
\item $x_{*}^\Delta(t)=
H^i_{({\psi^i})^{\sigma}}(t,x_{*}(t),u^1_{*}(t),\ldots,
u^N_{*}(t), (\psi^i)^{\sigma}(t))$;

\item $x_{*}(a)=x_a$;

\item $\displaystyle ({\psi^{i}})^\Delta(t)= -
H_{x}^i(t,x_{*}(t),u^1_{*}(t),\ldots, u^N_{*}(t),
(\psi^i)^{\sigma}(t)) $
$$ \qquad - \ \sum_{\stackrel{j=1}{j \neq i}}^N
\ H^i_{u_j}(t,x_{*}(t),u^1_{*}(t),\ldots, u^N_{*}(t),
(\psi^i)^{\sigma}(t))\cdot {\gamma^{j}_*}_x (t,x_a,x_{*}(t))$$
(${\gamma^{j}_*}_x$ denotes the partial derivative of
${\gamma^{j}_*}$ with respect to $x$);

\item $\psi^i(b)=0$;

\item $\displaystyle H^i_{u^i}(t,x_{*}(t),u^1_{*}(t),\ldots,
u^N_{*}(t), ({\psi^i})^{\sigma}(t))= 0$;
\end{enumerate}
for all $i=1,2, \ldots, N$ and $t \in \mathbb{T}^k$.
\end{theorem}

\begin{proof} Differently from Theorem~\ref{Theo1},
now the controls $u^i$ depend explicitly on the state variable
$x$. Suppose that the decision $N$-tuple $(u_{*}^{1},
\ldots,u_{*}^{N})$ is a MPS  Nash-equilibrium of \ $\Gamma_N$, and
$x_{*}$ is the associated state trajectory. Fix $i \in
\{1,2,\ldots, N\}$. The same reasoning as used in the proof of
Theorem~\ref{Theo1} permit to conclude that $(x_*(\cdot),
u_*^i(\cdot))$ is a weak local minimizer of problem
(\ref{problem}), where we suppose now that the controls are MPS.
In the following we will prove that
\begin{equation}
\label{min-controls=} \min_{all \  admissible \  MPS \
controls}\widetilde{J}^i(u^i) \ \ = \ \ \min_{all \ admissible\ OL
\ controls}\widetilde{J}^i(u^i) \, .
\end{equation}
An OL control can be considered as a MPS control, so it is
clear that
$$
\min_{all \  admissible \  MPS \ controls}\widetilde{J}^i(u^i) \ \
\leq \ \ \min_{all \ admissible\  OL \
controls}\widetilde{J}^i(u^i) \, .
$$
Since $u_{*}^{i}(t)=\gamma^{i}_*(t,x_a, x(t))$ is a MPS
Nash-equilibrium control, by the assumption of admissibility, the
equations
$x^{\Delta}(t)=\widetilde{f}(x(t),\gamma^{i}_*(t,x_a, x(t)))$
and $x(a) =x_a$
define a unique trajectory $x_*$. With this trajectory we define
now the OL control
$$
v^i_*(t):=\gamma^{i}_*(t,x_a, x_*(t)).
$$
Notice that
$$
\widetilde{J}^i(v_*^i)= \min_{all \  admissible \  MPS \
controls}\widetilde{J}^i(u^i) \, ,
$$
and hence equality (\ref{min-controls=}) holds.
Observe also that $(x_*(\cdot), v^i_*(\cdot))$ is a weak local
minimizer of problem (\ref{problem}). Applying
Theorem~\ref{PMP-time-scales} to $\mathcal{H}^i$, we conclude that
there exists a delta differentiable function
$\psi^i : \mathbb{T}\rightarrow \mathbb{R}^n$
on $\mathbb{T}^k$ such that
\begin{multline*}
\begin{split}
({\psi^{i}})^\Delta(t)  &=  -
\mathcal{H}^i_{x}(t,x_*(t),v^i_*(t),(\psi^i)^{\sigma}(t))\\
 &=  - H_{x}^i(t,x_{*}(t),u^1_{*}(t),\ldots, u^N_{*}(t),
(\psi^i)^{\sigma}(t))
\end{split} \\
 - {\sum_{\stackrel{j=1}{j \neq i}}^N} \
H^i_{u_j}(t,x_{*}(t),u^1_{*}(t),\ldots, u^N_{*}(t),
(\psi^i)^{\sigma}(t))\cdot {\gamma^{j}_*}_x (t,x_a,x_{*}(t)) \, .
\end{multline*}
The other conditions are obtained in a similar way.
\end{proof}


\section*{Acknowledgments}

This work was partially supported by the \emph{Portuguese
Foundation for Science and Technology} (FCT), through the R\&D
unit \emph{Center for Research in Optimization and Control} (CEOC)
of the University of Aveiro (http://ceoc.mat.ua.pt), cofinanced by
the European Community Fund FEDER/POCI 2010. The authors are
grateful to Gerhard Jank for all his availability during a
one-year stay at the University of Aveiro, and for having shared
with them his expertise and personal notes \cite{Jank} on
dynamic games.



\end{document}